\documentclass[12 pt]{article}
\usepackage{amsfonts}
\usepackage{amsrefs}
\usepackage{mathrsfs,amsmath}
\usepackage{amssymb}
\usepackage{graphicx}
\usepackage{wrapfig}
\usepackage{geometry}
\usepackage{amsthm}
\usepackage{commath}
\usepackage{color}
\newtheorem{theorem}{Theorem}[section]
\newtheorem{corollary}{Corollary}[theorem]
\newtheorem{lemma}[theorem]{Lemma}

\newtheorem*{remark*}{Remark}
\newtheorem*{theorem*}{Theorem}

\begin{document}

\title{Shapes of Centrally Symmetric Octahedra with Prescribed Cone-Deficits}
%

%                                     also used for the TOC unless
%                                     \toctitle is used
%
\author{Zili Wang}
\date{}
%

%
%%%% list of authors for the TOC (use if author list has to be modified)
%

\maketitle              % typeset the title of the contribution

\section{Introduction}
The surface of a polyhedron is a Euclidean cone metric on the sphere, a metric that is locally Euclidean except at finitely many points, namely its vertices. Each vertex has a neighborhood isometric to a Euclidean cone, which can be obtained by gluing two bounding rays of a sector that make an angle $\theta$. We will call $\theta$ the \textit{cone-angle}, and $2\pi-\theta$ the \textit{cone-deficit} associate to this vertex. The sum of the cone-deficits associated to all vertices is $4\pi$ by the Gauss-Bonnet Theorem.

A polyhedron is convex if the cone-angle associated to each vertex is strictly less than $2\pi$. It is centrally symmetric if it is symmetric about the origin. That is, a point $x$ lies in this polyhedron if and only if its \textit{antipodal point},$-x$, also lies in it. We call the map $x\mapsto -x$ the \textit{antipodal map}. Thus, a centrally symmetric polyhedron have an even number $2n$ of vertices, and the cone-deficits are determined by specifying $n$ positive numbers that add up to $2\pi$. 

In \cite{Thu}, Thurston considered the space of convex polyhedra with $n$ vertices and prescribed cone-deficits. He built local coordinate charts from the space into $\mathbb{C}^{n-2}$ based on a decomposition of the surface of each polyhedron into triangles. Then he showed that the surface area function in these coordinates is a Hermitian form of signature $(1,n-3)$. This gives a natural metric on the space of such polyhedra whose total surface area are equal to $1$. With respect to this metric, the space is locally isometric to the complex hyperbolic space of dimension $n-3$.

In this work, we focus on centrally symmetric octahedra with prescribed cone-deficits and labeled vertices. Let $\delta_1,\delta_2,\delta_3$ be three positive numbers that sum to $2\pi$. We consider the collection $\mathcal{C}(\delta_1,\delta_2,\delta_3)$ of centrally symmetric octahedra with cone-deficits $\delta_1,\delta_2,\delta_3$ and of total surface area $1$, in which two octahedra are equivalent if there is an isometry between them that respects vertex labels. Analogous to Thurston's description, we show that:

\begin{theorem*} There is a natural metric on $\mathcal{C}(\delta_1,\delta_2,\delta_3)$, with respect to which it is isometric to a real hyperbolic ideal tetrahedron with dihedral angles $\frac{\delta_1}{2},\frac{\delta_2}{2}$ and $\frac{\delta_3}{2}$.
\end{theorem*}   

The outline of the proof is as follows:

In Section 2.1, we describe a way to decompose the surface of any centrally symmetric octahedron into twelve parallelograms. The decomposition uses all the vertices together with eight extra points and is invariant under the antipodal map.

In Section 2.2, we study the space of centrally symmetric octahedra with labeled vertices and prescribed cone-deficits. Based on the decomposition in Section 2.1, we build one coordinate chart that identifies this space with the positive orthant in $\mathbb{R}^4$.

In Section 2.3, using these coordinates, we show that the surface area function is a quadratic form of signature $(1,3)$. This determines a metric with respect to which $\mathcal{C}(\delta_1,\delta_2,\delta_3)$ is locally isometric to the real hyperbolic space $\mathbb{H}^3$. We also show that the boundary of $\mathcal{C}(\delta_1,\delta_2,\delta_3)$, which consists of degenerate octahedra, has four geodesic hyperplanes, each three of which meet at an ideal point.

In Section 2.4, we compute the dihedral angles of $\mathcal{C}(\delta_1,\delta_2,\delta_3)$ by computing the angles between the normal vectors to its bounding hyperplanes.

\subsection*{Acknowledgement}
I would like to thank my adviser Richard Schwartz, from whom I learned this topic and got many helpful feedbacks on the ideas in and the structure of this article.  

\section{Proof of the Theorem}

\subsection{Decomposition of the Surface into Parallelograms}

Let $\Sigma$ be the surface of a centrally symmetric octahedron with cone-deficits $\delta_1,\delta_2$ and $\delta_3$ and corresponding cone-angles $\theta_1,\theta_2$ and $\theta_3$. The faces of the octahedron form a division of $\Sigma$ into eight triangles. We will find a decomposition of $\Sigma$ into twelve parallelograms that uses all the vertices of $\Sigma$ together with eight extra points, one in the interior of each triangle. The decomposition is also invariant under the antipodal map. In the following figure we draw two schematic pictures for this decomposition, in each of which the reader can see four parallelograms in full and another four in half. In the left picture, we color the edges of the parallelograms in the decomposition such that those with the same color have equal length. In the right one, we color the vertices of $\Sigma$ and the parallelograms such that a vertex and a parallelogram get the same color if an angle of the parallelogram is half of the cone-deficit(or cone-angle) associated to that vertex.

\begin{figure}[h]
\centering
\includegraphics[width=0.655\textwidth]{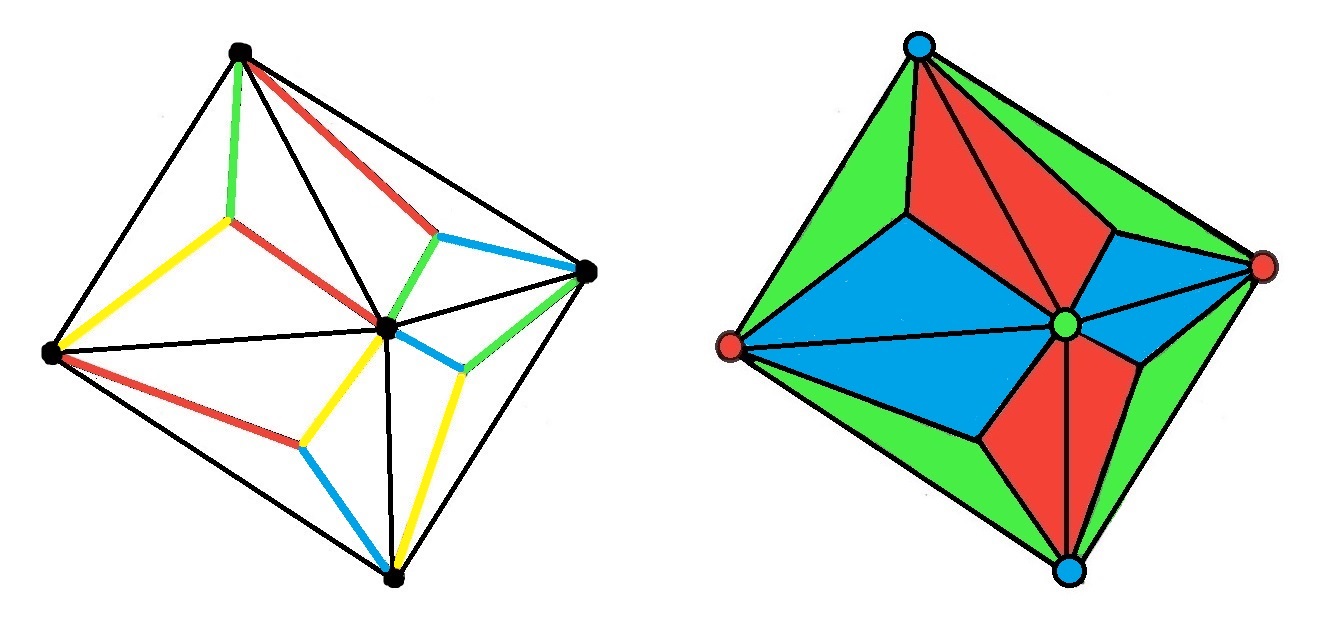}
\label{fig:1}
\end{figure} 

We need to introduce some notations before explaining how the decomposition works. Let $v_1,v_2,v_3,v_{1'},v_{2'},v_{3'}$ be the vertices of $\Sigma$, such that $v_{i'}$ is the antipodal point of $v_i$, and the cone-deficit associate to them is $\delta_i$. Let $T_1, T_2, T_3$ and $T_4$ be the four triangles incident to $v_1$, and let $T_{1'}, T_{2'}, T_{3'}, T_{4'}$ be their images under the antipodal map. Figure 1 shows two side views of $\Sigma$ when $v_1$ and $v_2$ are facing the reader. 

We denote by $\omega_{ij}$ the angle at $v_j$ in $T_i$. Similarly, we also have $\omega_{ij'}$, $\omega_{i'j}$ and $\omega_{i'j'}$. Some of these angles are marked in Figure 1. Note that $\omega_{ij}=\omega_{i'j'}$.
\begin{figure}[h]
\centering
\includegraphics[width=0.7\textwidth]{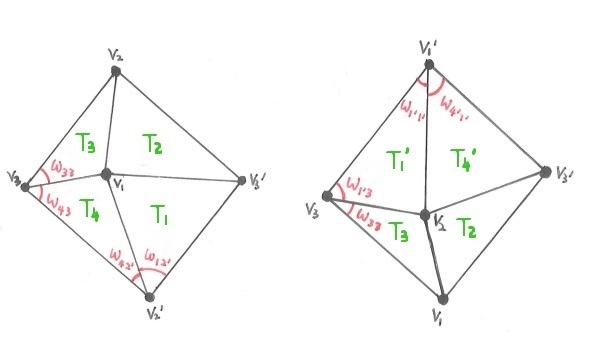}
\caption{Two side views of $\Sigma$ when $v_1$ and $v_2$ are facing the reader. }
\label{fig:1}
\end{figure} 

Now we will find two of the eight extra points in the parallelogram-decomposition. Consider the following two angles $$\alpha=\frac{\omega_{12'}+\omega_{42'}+\omega_{43}+\omega_{33}-\pi}{2}$$
$$\beta=\frac{\omega_{32}+\omega_{22}+\omega_{23'}+\omega_{13'}-\pi}{2}$$

\begin{lemma}
There is a unique point $O_3$ in the interior of $T_3$ such that $\angle O_3v_3v_2=\alpha$ and $\angle O_3v_2v_3=\beta$. In addition, there is a unique point $O_1$ in the interior of $T_1$ such that $\angle O_1v_{2'}v_{3'}=\alpha$ and $\angle O_1v_{3'}v_{2'}=\beta$.
\end{lemma}

\begin{proof}
We prove the statement for $O_3$ and that for $O_1$ will follow from symmetry.

Note that $v_2,v_3,v_{2'}$ and $v_{3'}$ are coplanar. Hence they are the vertices of a planar quadrilateral, which must be a parallelogram by symmetry. Thus, we have $$2\alpha=\omega_{12'}+\omega_{42'}+\omega_{43}+\omega_{33}-\pi>0$$

Similarly, $v_1,v_3,v_{1'}$ and $v_{3'}$ are vertices of a parallelogram, so $$\omega_{33}+\omega_{1'3}+\omega_{1'1'}+\omega_{4'1'}-\pi>0$$

and their sum 
\begin{equation*}
\begin{split}
&\omega_{12'}+\omega_{42'}+\omega_{43}+\omega_{33}-\pi+\omega_{33}+\omega_{1'3}+\omega_{1'1'}+\omega_{4'1'}-\pi\\
=&2\omega_{33}+ \omega_{12'}+\omega_{42'}+\omega_{43}+\omega_{13'}+\omega_{11}+\omega_{41}-2\pi\\
=&2\omega_{33}+(\omega_{42'}+\omega_{43}+\omega_{41})+(\omega_{12'}+\omega_{13'}+\omega_{11})-2\pi\\
=&2\omega_{33}+2\pi-2\pi\\
=&2\omega_{33}
\end{split}
\end{equation*}
in which the third equality holds because the sums in the brackets are the angle-sums of $T_4$ and $T_1$.

The above equation and two inequalities imply $$0<\alpha<\omega_{33}$$

Similarly, from $$2\beta=\omega_{32}+\omega_{22}+\omega_{23'}+\omega_{13'}-\pi>0$$
$$\omega_{32}+\omega_{1'2}+\omega_{1'1'}+\omega_{2'1'}-\pi>0$$

and their sum equals $2\omega_{32}$, we conclude that $$0<\beta<\omega_{32}$$

Given the ranges of $\alpha$ and $\beta$, there must be a unique point $O_3$ in the interior of $T_3$ such that $\angle O_3v_3v_2=\alpha$ and $\angle O_3v_2v_3=\beta$, as drawn in the following figure. 
\end{proof}

\begin{figure}[h]
\centering
\includegraphics[width=0.78\textwidth]{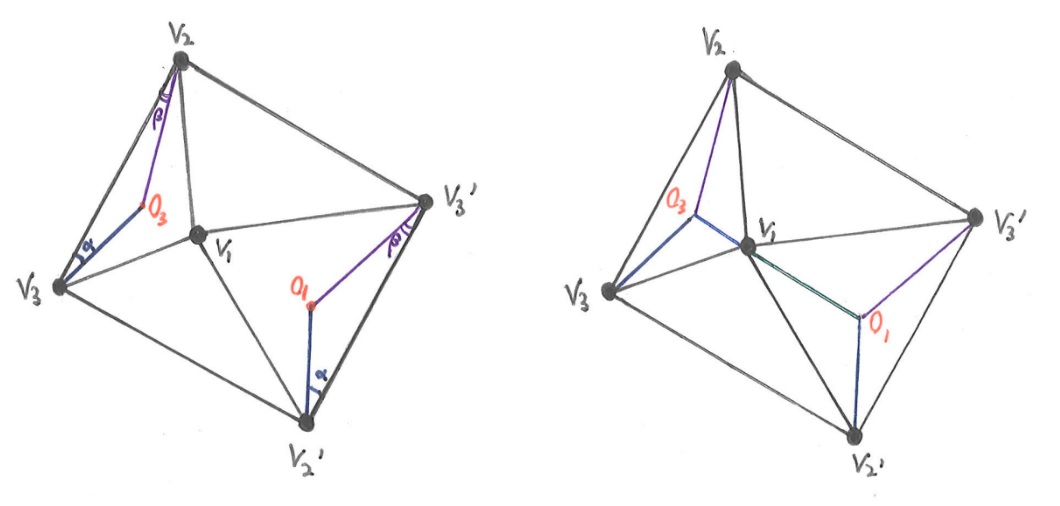}
\end{figure} 

Denote by $O_{1'}$ and $O_{3'}$ the antipodal points of $O_{1}$ and $O_{3}$. Let $O_{1},O_{1'},O_3$ and $O_{3'}$ be four of the eight extra points in the parallelogram-decomposition.

We join $O_1v_{1}$, $O_1v_{2'}$, $O_1v_{3'}$, $O_3v_{3}$, $O_3v_{2}$ and $O_3v_{1}$ by distance-minimizing geodesics, each of which is a line segment inside some triangle. Now we will find a subset of $\Sigma$ that is decomposable into parallelograms.

Let $A$ be the geodesic pentagon with vertices $v_3$, $O_3$, $v_1$, $O_1$ and $v_{2'}$, and $B$ be the geodesic pentagon with vertices $v_{3'}$, $O_1$, $v_1$, $O_3$ and $v_{2}$. Consider the union of $A$ and the image of $B$ under the antipodal map. This is a geodesic octagon in $\Sigma$, with vertices $O_1,v_1,O_3,v_3,O_{1'},v_{1'},O_{3'}$ and $v_{2'}$. We identify it with a planar simple octagon, as sketched in Figure 2. We will show how to decompose this planar octagon into parallelograms. We begin with two observations:

\textbf{Observation 1}: Since $\triangle v_{3'}O_1v_{2'}$ is congruent to $\triangle v_2O_3v_3$ (by ASA congruence condition), the lengths of the Euclidean segments $O_1v_{2'}$ and $O_3v_3$ are equal. So do the lengths of $O_1v_{3'}$ and $O_3v_2$ by antipodal symmetry. 

\textbf{Observation 2}: From the definition of the angle $\alpha$, we know $O_1v_{2'}$ and $O_3v_3$ are parallel. Similarly, from the definition of $\beta$ and antipodal symmetry, we know $O_1v_{3'}$ and $O_3v_2$ are parallel.

\begin{figure}[h]
\centering
\includegraphics[width=0.4\textwidth]{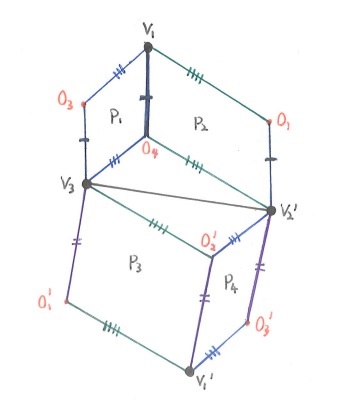}
\caption{This planar octagon is isometric to a subset of $\Sigma$. The segments in this picture with the same marking are parallel and have equal length.}
\label{fig:2}
\end{figure} 

\begin{lemma}
The octagon above has a decomposition into five parallelograms using its vertices and two extra points.
\end{lemma}

\begin{proof}
Let $\angle O_1v_1O_3$ and $\angle O_{1'}v_{1'}O_{3'}$ be the interior angles of the octagon at $v_1$ and $v_{1'}$. We first verify that they are indeed strictly less than $\pi$ as Figure 2 suggests. From Observation 1 and 2, we see that the Euclidean segments $O_1v_3$ and $O_{1'}O_{3'}$ have equal length. So the Euclidean triangle with vertices $O_1,v_1,O_3$ is congruent to the Euclidean triangle with vertices $O_{1'},v_{1'},O_{3'}$(by SSS congruence condition). Let $\psi$ be the magnitude of their angles at $v_1$ and $v_{1'}$. Then $\angle O_1v_1O_3$ must be $\psi$ or $2\pi-\psi$, and so is $\angle O_{1'}v_{1'}O_{3'}$. But given their sum is $\theta_1$ by construction, they must both be equal to $\psi$, and $\psi=\frac{\theta_1}{2}$.

Now since $O_1v_{2'}$ and $O_3v_3$ are parallel and have equal length by Observation 1 and 2, there is a unique point $O_4$ such that the quadrilateral with vertices $O_3,v_1,O_4$ and $v_3$ and the one with vertices $O_1,v_1,O_4$ and $v_{2'}$ are both parallelograms. They are marked as $P_1$ and $P_2$ in Figure 2. We see that $O_4$ is in the interior of $T_4$ since the three angles meeting at $O_4$ are all strictly less than $\pi$.

Similarly, there is a unique point $O_{2'}$ such that both the quadrilateral with vertices $O_{1'},v_{1'},O_{2'},v_{3}$ and the one with vertices $O_{3'},v_{1'},O_{2'},v_{2'}$ are parallelograms. They are marked as $P_3$ and $P_4$. Note that $O_{2'}$ is in the interior of $T_{2'}$.

We are left with the central quadrilateral with vertices $O_{4},v_{2'},O_{2'}$ and $v_{3}$. Since the length of its opposite sides are equal, it must be a parallelogram, denoted by $P_5$.  \end{proof}

Let $O_{4'}$ and $O_{2}$ be the antipodes of $O_{4}$ and $O_{2'}$. Then the complete list of the extra points in the decomposition is $\{O_1,O_{2},O_{3},O_{4},O_{1'},O_{2'},O_{3'},O_{4'}\}$.

To see that we indeed get a parallelogram-decomposition of $\Sigma$, notice that the geodesic quadrilateral on $\Sigma$ with vertices $O_{1'}, v_2, O_3$ and $v_3$ is also a parallelogram by Observation 1, denoted by $P_6$. The twelve parallelograms in the decomposition are hence given by $P_1,P_2,P_3,P_4,P_5,P_6$, together with their images under the antipodal map. 

\subsection{A Coordinate Chart for the Space of Centrally Symmetric Octahedra}

Let $a,b,c,d$ be the lengths of the segments $O_1v_{2'},O_1v_{3'},O_1v_{1}$ and $O_3v_1$. In this section, we show that the space of centrally symmetric octahedra with labeled vertices and prescribed cone-deficits can be identified with the positive orthant in $\mathbb{R}^4$ under the coordinate chart that maps $\Sigma$ to $(a,b,c,d)$.
 
\begin{lemma} 
\label{lem 2}
The side lengths of each parallelogram in the decomposition of $\Sigma$ are from the list $\{a,b,c,d\}$, and one of the angles between two sides is from the list $\{\frac{\delta_1}{2},\frac{\delta_2}{2},\frac{\delta_3}{2}\}$.
\end{lemma} 

\begin{proof}
The side lengths of $P_6$ are $a$ and $b$ by construction. From the parallelogram-decomposition of the octagon in Figure 2, we can find the side lengths of parallelograms from $P_1$ to $P_5$.

For the second statement, we compute $\angle O_{1'}v_2O_3$ in $P_6$. The angles in the remaining parallelograms can be computed similarly.
\begin{equation*}
\begin{split}
\angle O_{1'}v_2O_3&=\alpha+\beta\\ 
&=\frac{\omega_{12'}+\omega_{42'}+\omega_{43}+\omega_{33}-\pi}{2}+\frac{\omega_{32}+\omega_{22}+\omega_{23'}+\omega_{13'}-\pi}{2} \\
&=\frac{4\pi-\theta_1}{2}-\pi\\
&=\frac{\delta_1}{2}
\end{split}
\end{equation*}
in which the third equality is because $\omega_{12'}+\omega_{13'}+\omega_{42'}+\omega_{43}+\omega_{33}+\omega_{32}+\omega_{22}+\omega_{23'}$ sums up all the angles in $T_1,T_4,T_3$ and $T_2$ except for their vertex-angles at $v_1$.

By going through similar calculation with other parallelograms, we are able to verify the properties of this parallelogram-decomposition claimed in the pictures at the beginning of Section 2.1.
\end{proof}

\begin{corollary}
\label{cor 1}
The area of $\Sigma$ is twice of $$(ab+cd)\sin\frac{\delta_1}{2}+(ac+bd)\sin\frac{\delta_2}{2}+(ad+bc)\sin\frac{\delta_3}{2}$$
\end{corollary}
 
Since the parallelogram-decomposition of a centrally symmetric octahedron only depends on its structure and vertex-labels, each octahedron can be associated to a quadruple $(a,b,c,d)$ of positive numbers in a unique way. Conversely, given any ordered quadruple of positive numbers, we can construct six pairs of parallelograms based on Lemma \ref{lem 2}, and glue them to construct a unique octahedron up to isometry that respects vertex-labels. This gives an identification of the space of centrally symmetric octahedra with labeled vertices and prescribed cone-deficits with the positive orthant in $\mathbb{R}^4$. In addition, we see that $\mathcal{C}(\delta_1,\delta_2,\delta_3)$ is the set of $(a,b,c,d)$ in this orthant such that
$$2\left[(ab+cd)\sin\frac{\delta_1}{2}+(ac+bd)\sin\frac{\delta_2}{2}+(ad+bc)\sin\frac{\delta_3}{2}\right]=1$$

\subsection{The Space $\mathcal{C}(\delta_1,\delta_2,\delta_3)$ Is an Ideal Tetrahedron}

In this section, we show that $\mathcal{C}(\delta_1,\delta_2,\delta_3)$ is locally isometric to $\mathbb{H}^3$, and is in fact an ideal tetrahedron by looking at its boundary. 

The surface area function in Corollary~\ref{cor 1} is a quadratic form of four variables. We compute its signature below. 

To make the computations neater, we introduce the following two notations:
$$S_i=\sin\frac{\delta_i}{2},\quad C_i=\cos\frac{\delta_i}{2}$$.

\begin{lemma}
The area function has signature $(1,3)$.
\end{lemma}

\begin{proof}
The symmetric matrix associated to this function is
\[
\begin{bmatrix}
    0                       & S_1 & S_2 & S_3 \\
    S_1  & 0                      & S_3 & S_2 \\
    S_2  & S_3 & 0                      & S_1 \\
    S_3  & S_2 & S_1 & 0
\end{bmatrix}
\]
and the characteristic polynomial is given by 
\[
\begin{split}
&x^4-2({S_1}^2+{S_2}^2+{S_3}^2)x^2-8{S_1}{S_2}{S_3}x+{S_1}^4+{S_2}^4+{S_3}^4-2{S_1}^2{S_2}^2-2{S_1}^2{S_3}^2 -2{S_2}^2{S_3}^2
 \end{split}
\]

We can then factor this polynomial to get
$$(x-{S_1}-S_2-S_3)(x +S_1+S_2-S_3)(x+S_1-S_2+S_3)(x-S_1+S_2+S_3)$$

Hence, the roots are

$x_1=S_1+S_2+S_3$,  $x_2=S_3-S_2-S_1$, $x_3=S_2-S_1-S_3$ and $x_4=S_1-S_2-S_3$.

Since $\delta_1,\delta_2,\delta_3>0$ and $\frac{\delta_1}{2}+\frac{\delta_2}{2}+\frac{\delta_3}{2}=\pi$, we have
\[
\begin{split}
x_2&=\sin\frac{\delta_1+\delta_2}{2}-S_1-S_2\\
&=S_1C_2+S_2C_1-S_1-S_2\\
&=S_1(C_2-1)+S_2(C_1-1)\\
&<0
\end{split}
\]

By symmetry, we have $x_3,x_4<0$. Clearly, $x_1>0$. So the form has signature  $(1,3)$. \end{proof}

The space of vectors of length $1$ in a quadratic form of signature $(1, 3)$ is isometric to the real hyperbolic space $\mathbb{H}^3$. In addition, $\mathcal{C}(\delta_1,\delta_2,\delta_3)$ is bounded by four geodesic hyperplanes. For instance, the set of points with the last coordinate vanish $(d=0)$ is the fixed point set of the isometry $(a,b,c,d)\mapsto(a+2dC_3,b+2dC_2,c+2dC_1,-d)$. Finally, every three hyperplanes intersect at an ideal point with three of the coordinates vanish. Therefore, $\mathcal{C}(\delta_1,\delta_2,\delta_3)$ is the interior of a real hyperbolic ideal tetrahedron. The points on the boundary are those octahedra with one or two coordinates vanish and degenerate to pillowcases of centrally symmetric hexagons or parallelograms.

We may also consider the ``unlabeled space'', in which two octahedra are equivalent if there is an isometry between them that respect cone-deficit values, not necessarily the vertex-labels. This space is just $\mathcal{C}(\delta_1,\delta_2,\delta_3)$ when all cone-deficit values are distinct. Otherwise, $\mathcal{C}(\delta_1,\delta_2,\delta_3)$ has a nontrivial symmetry group, the elements in which map every octahedron to another isometric one. For example, if $\delta_2=\delta_3$, then the symmetry group has two generators. We can obtain one of them by interchanging $a$ and $b$, and the other one by interchanging $c$ and $d$. So the unlabeled space is the quotient of $\mathcal{C}(\delta_1,\delta_2,\delta_3)$ by the dihedral group $D_2$. Finally, if $\delta_1=\delta_2=\delta_3$, the symmetry group is the permutation group on $\{a,b,c,d\}$, so the unlabeled space is the quotient of $\mathcal{C}(\delta_1,\delta_2,\delta_3)$ by $S_4$.

\subsection{Dihedral angles of $\mathcal{C}(\delta_1,\delta_2,\delta_3)$}

We have shown that $\mathcal{C}(\delta_1,\delta_2,\delta_3)$ is an ideal tetrahedron, bounded by four geodesic planes. In this section, we will compute its dihedral angles by finding a normal vector to each plane. 

The area function is the diagonal part of an inner product on $\mathbb{R}^4\times\mathbb{R}^4$, given by 
\[
\begin{split}
&(a,b,c,d)\ast(a',b',c',d')\\
=&(ab'+a'b+cd'+c'd)S_1+(ac'+a'c+bd'+b'd)S_2+(ad'+a'd+bc'+b'c)S_3
\end{split}
\]

From this formula, a normal vector to the plane $a=0$ is a vector $(n_a,n_b,n_c,n_d)$ satisfying
$$b(n_aS_1+n_cS_3+n_dS_2)+c(n_aS_2+n_bS_3+n_dS_1)+d(n_aS_3+n_bS_2+n_cS_1)=0$$
for all $b,c,d$.

Given the trigonometric formula $$S_i=S_jC_k+S_kC_j$$
for mutually distinct $i,j,k\in\{1,2,3\}$, we observe that in order to make the expression in each bracket vanish, we can take
$$(n_a,n_b,n_c,n_d)=(1,-C_1,-C_2,-C_3)$$

Similarly, we can find a normal vector $(m_a,m_b,m_c,m_d)=(-C_1,1,-C_3,-C_2)$ to the plane $b=0$.

A straightforward calculation shows that both vectors have the same length, and their inner product divided by their length square is $-C_1$. Therefore, the dihedral angle between the planes $a=0$ and $b=0$ is $\frac{\delta_1}{2}$. Some details of this calculation are given in the Appendix.

We get the rest of dihedral angles by symmetry in the area function $(ab+cd)S_1+(ac+bd)S_2+(ad+bc)S_3$. For instance, by interchanging $b$ and $c$, $\delta_1$ and $\delta_2$, we find the dihedral angle between the planes $a=0$ and $c=0$ is $\frac{\delta_2}{2}$. This proves the main theorem of this work.

\begin{corollary}
The volume of $\mathcal{C}(\delta_1,\delta_2,\delta_3)$ is $L(\frac{\delta_1}{2})+L(\frac{\delta_2}{2})+L(\frac{\delta_3}{2})$, where $L(x)=-\int_0^x log(2\sin\theta)d\theta$ is the Lobachevsky function.
\end{corollary}

\section*{Appendix: Calculating the Angle Between Two Vectors}

Here we give more details in computing the angle between the vectors $(1,-C_1,-C_2,-C_3)$ and $(-C_1,1,-C_3,-C_2)$ for verification. The two vectors have the same length by symmetry in the expression of the quadratic form.

Their inner product is
\[
\begin{split}
&(-C_1,1,-C_3,-C_2)\ast(1,-C_1,-C_2,-C_3)\\
=&S_1({C_1}^2+1+{C_3}^2+{C_2}^2)+2S_2(C_1C_2-C_3)+2S_3(C_1C_3-C_2)\\
=&S_1({C_1}^2+1)+(S_1{C_3}^2+S_3C_1C_3)+(S_1{C_2}^2+S_2C_1C_2)\\
&-2(S_3C_2+S_2C_3)+(S_2C_1C_2+S_3C_1C_3)\\
=&S_1({C_1}^2+1)+S_2C_3+S_3C_2-2(S_3C_2+S_2C_3)+(S_2C_2+S_3C_3)C_1\\
=&S_1({C_1}^2+1)-S_1+(S_2C_2+S_3C_3)C_1\\
=&(S_1C_1+S_2C_2+S_3C_3)C_1
\end{split}
\]

On the other hand, the length square of $(1,-C_1,-C_2,-C_3)$ is 
\[
\begin{split}
&(1,-C_1,-C_2,-C_3)\ast(1,-C_1,-C_2,-C_3)\\
=&2S_1(C_2C_3-C_1)+2S_2(C_1C_3-C_2)+2S_3(C_1C_2-C_3)\\
=&(S_1C_2C_3+S_3C_1C_2)+(S_1C_2C_3+S_2C_1C_3)(S_3C_1C_2+S_2C_1C_3)\\
&-2S_1C_1-2S_2C_2-2S_3C_3\\
=&S_2C_2+S_3C_3+S_1C_1-2S_1C_1-2S_2C_2-2S_3C_3\\
=&-(S_1C_1+S_2C_2+S_3C_3)
\end{split}
\]

The quotient of these two expressions is $-\cos\frac{\delta_1}{2}$ as expected.

\end{document}